\documentclass{amsart}

\usepackage
  { amsthm
  , amssymb
  , amsmath
  , amsrefs
  , MnSymbol
  , aliascnt
  , hyperref
  , tikz
  , xcolor
  , pgfplots
  }
\usepackage{theorems}

\NewDocumentCommand\normclsr{m}{\left\llangle #1\right\rrangle}
\NewDocumentCommand\fd{mm}{\frac{\partial #1}{\partial #2}}
\NewDocumentCommand\comm{m}{\left[#1,#1\right]}
\NewDocumentCommand\Q{}{\mathbb Q}
\DeclareMathOperator{\Aut}{Aut}
\DeclareMathOperator{\Ab}{Ab}
\title{More residually solvable one-relator groups}
\author[]{Lucy Koch-Hyde}
\address{Department of Mathematics, CUNY Graduate Center, New York,
NY 10016}  \email{Lucy.Koch-Hyde96@gc.cuny.edu}
\author[]{Siobh\'an O'Connor}
\address{Department of Mathematics, CUNY Graduate Center, New York,
NY 10016}  \email{Siobhan.OConnor75@gc.cuny.edu}
\author[]{\'Eamonn Olive}
\email{ejolive97@gmail.com}
\date{June 2026}

\begin{document}

\begin{abstract}
    We say a word in a free group is ``partially positive'' if there is at least one generator which appears in the word, but does not appear with negative exponent.
    We show that one-relator groups with a partially positive relator are residually solvable.
    This extends a theorem of Baumslag showing one-relator groups with positive relators are residually solvable.
\end{abstract}

\maketitle

\section{Introduction}

We will begin with the seminal theorem of Baumslag:

\begin{theorem}[\cite{baumslag_positive_1971}*{Thm.\ 1}]\label{thm:Baumslag}
    Let $F_X$ be a free group.
    If $w$ is a word in the ``positive submonoid'' of $F_X$, i.e.\ the submonoid generated by $X$ (and not $X^{-1}$), then the group $G=F_X/\normclsr{w} = \left\langle X\mid w=1\right\rangle$ is residually solvable.
\end{theorem}

We will define residual solvability, but the internal mechanics are of less importance than the fact it is an isomorphism invariant.

\begin{definition}
    A group $G$ is residually solvable if its derived subgroups converge to the trivial group.
    That is,
    \[
    \{1\} = G^{(\omega)} :=\bigcap^\infty_{i=0} G^{(i)}
    \]
    where $G^{(0)}:=G$ and $G^{(i+1)}:=\comm{G^{(i)}}$.
\end{definition}

What one ought to take away is that residual solvability is a condition about commutator subgroups (and so are solvability and some related notions we will explore later).

If we want to know how strong an isomorphism invariant is, a good question to ask is how probable it is that two random groups are distinguished by the isomorphism invariant.
To know this we can ask what the probability that a random group is residually solvable is.

\begin{definition}
    Let $\mathbb P_r^n(P)$ be the probability that for a uniformly sampled word $w$ with $|w|\leq n$, $F_r/\normclsr{w}$ has property $P$.
    We define $\mathbb P_r^\infty(P):=\lim_{n\rightarrow \infty}\mathbb P_r^n(P)$.
\end{definition}

This of course sets us up for the question:

\begin{question}\label{q:rsprob}
    What is the probability that a random one-relator group is residually solvable?
\end{question}

This question is a more general version of a question in \cite{bencsath_residually_2010} which asks simply whether $\mathbb P_r^\infty(\text{residual solvability})=1$.

Now, \autoref{thm:Baumslag} tells us that positive one-relator groups are residually solvable, so it gives us a lower bound on the probability function $\mathbb P_r^n(\text{residually solvable})$.
The number of positive words grows as $r^n$, and the total number of reduced words grows as $2r(2r-1)^n$, meaning that the probability that a random word is positive converges to zero exponentially fast (so long as the rank is at least 2, which we will generally assume from now on).

One can expand the scope of \autoref{thm:Baumslag} using a trick.
Applying a free group automorphism to the relator yields an isomorphic group, so it preserves isomorphism invariants.
Motivated by this fact, we define potentially positive words to be those which are automorphic to positive words.
Potentially positive words were introduced in the 2002 edition of \cite{kapovich_combinatorial_2026}.
In the current edition, they are addressed in problem (F33):

\begin{question}[\cite{kapovich_combinatorial_2026}*{(F33)}]
    Let potential positivity be defined as above.
    \begin{enumerate}
        \item[(a)] Is the property of being potentially positive algorithmically recognizable?
        \item[(c)]\label{pp growth} Let $P(r,n)$ denote the number of potentially positive elements of length $n$ in $F_r$. What is the growth of $P(r,n)$ as a function of $n$, with $r$ fixed?
    \end{enumerate}
\end{question}

Goldstein first provided an algorithm in rank 2 in \cite{goldstein_algorithm_2006}.
\cite{koch-hyde_complexity_nodate} and \cite{oconnor_algorithmic_2026} improve its time complexity, ultimately giving a linear time algorithm.%
\footnote{
Asymptotic complexity can be thought of as an objective measure of descriptive power.
The faster or more space efficient an algorithm, the more you have said about a problem.
}
\cite{dinowitz_growth_2025} uses Goldstein's description of potentially positive elements to solve \hyperref[pp growth]{(c)} for $F_2$.
Scant progress has been made on these problems in ranks higher than 2.

While we do not have a full description of potential positivity in ranks 3 and up, we can still make observations about this property which expand the scope of \autoref{thm:Baumslag}.
One such observation is the following:
\begin{observation}\label{obs:obv}
    Any word $w\in F_r$ in which all but one generator appears positively is potentially positive.
\end{observation}

One could hope to expand the scope even further.

\begin{question}[Vladimir Shpilrain, personal communication]\label{q:Shpilrain's question}
    A word $w\in F_X$ is {\it partially positive} if there is some $a\in X$ such that $w$ can be written with only positive exponents on $a$, and $a$ appears at least once in $w$.
    Is there a partially positive word which is not potentially positive?
\end{question}

Restricted to rank 2 the answer is ``no'' per \autoref{obs:obv}, however, for higher ranks the answer seems almost certainly to be ``yes''.
Nevertheless, actual proof for higher ranks is slippery.
We would likely have to prove a particular word (outside of the commutator subgroup) is not potentially positive.
Doing so is not easy even if one has very good evidence to believe a particular word cannot be potentially positive.
One way we might try to resolve this is to turn \autoref{thm:Baumslag} on its head.
We can look for partially positive words which relate non-residually solvable groups.

However, there is an obstacle to producing such examples: they do not exist.

\begin{theorem}[Main result]
    Let $F_X$ be a free group.
    If $w\in F_X$ is a partially positive word then the group $G=F_X/\normclsr{w} = \left\langle X\mid w=1\right\rangle$ is residually solvable.
\end{theorem}

Instead of answering \autoref{q:Shpilrain's question}, we have a pretty extensive generalization of \hyperref[thm:Baumslag]{Baumslag's result}.

\section{Methods}

As is often the case, we will prove a slightly stronger theorem than our stated main result.

\begin{theorem}\label{thm:real main}
    Let $F_X$ be a free group
    and $w\in F_X$ be a partially positive word.
    \begin{itemize}
        \item If $w$ is not a proper power, then $F_X/\normclsr{w}$ is residually $\Q$-solvable.
        \item If $w$ is a proper power, then $F_X/\normclsr{w}$ is free-by-solvable.
    \end{itemize}
\end{theorem}

Residual $\mathbb Q$-solvability and free-by-solvability are both generalizations of residual solvability.
We will start by defining free-by-solvable groups:

\begin{definition}
    A group $G$ is free-by-solvable, if its derived series contains a free group.
    That is, there is some $n$ such that $G^{(n)}$ is free.

    Equivalently, $G$ is free-by-solvable if there is a homomorphism $\varphi : G\rightarrow Q$ where $Q$ is a solvable group and $\ker(\varphi)$ is free.
\end{definition}

Since a free group is itself residually solvable, every free-by-solvable group is as well.

When we look at residually solvable groups, this generalization feels pretty natural.
There are many groups which are almost solvable but are obstructed because they contain a non-abelian free subgroup.
For example, by the Freiheitssatz, all one-relator groups with more than two generators contain non-abelian free subgroups.

\begin{definition}
    A group $G$ is residually $\Q$-solvable if its rational derived series converges to the trivial group.
    That is,
    \[
    \{1\} = G^{(\omega)}_\Q :=\bigcap^\infty_{i=0} G^{(i)}_{\Q}
    \]
    where $G^{(0)}_\Q:=G$ and $G^{(i+1)}_\Q$ is the subgroup of all elements which have a non-zero power in $\comm{G^{(i)}_\Q}$.
\end{definition}

It should be clear that $G^{(\omega)}_\Q \supseteq G^{(\omega)}$;
thus, residual $\Q$-solvability implies residual solvability.

The thrust of this paper comes from results in \cite{linton_residually_2025}.
In particular, the following:

\begin{theorem}[\cite{linton_residually_2025}*{Cor.\ 1.2}]\label{thm:linton main}
    Let $F$ be a free group, $w\in F$ a word, and $G=F/\normclsr{w}$ a one-relator group.
    The following are equivalent:
    \begin{enumerate}
        \item $G$ is residually $\mathbb{Q}$-solvable.
        \item $G_\mathbb{Q}^{(n)}$ is a free group for some integer $n\geq 0$.
        \item For all words $r\in F$ such that $|r|\leq|w|$, and all integers $k\geq 1$, if $w\in r^k \comm{\normclsr{r}}$, $w$ is conjugate to $r$ or $r^{-1}$.
    \end{enumerate}
\end{theorem}

In \cite{linton_residually_2025}, Linton is able to use this result to create an algorithm that decides whether a one-relator group is residually $\Q$-solvable by checking all possible values for $r$ up to the length of the word.

We will also require the following lemma which is used implicitly in \cite{linton_residually_2025}:

\begin{lemma}\label{thm:linton torsion}
    Let $F$ be a free group. Let $w\in F$ be a word which is not a proper power.
    If $F/\normclsr{w}$ is residually $\Q$-solvable then $F/\normclsr{w^n}$ is free-by-solvable for all $n$.
\end{lemma}
\begin{proof}
    By the proof of \cite{linton_residually_2025}*{Lemma 4.2}, there exists a solvable group $Q$ and a homomorphism $\varphi: F/\normclsr{w^n}\rightarrow Q$ such that $\ker(\varphi)$ is the free product of cyclic groups.
    Assume without loss of generality that $\varphi$ is onto.

    By the Kurosh subgroup theorem (\cite{kuhn_subgroup_1952}*{Thm.\ 3.01}) $\comm{\ker(\varphi)}$ is free.
    Since the commutator subgroup is a characteristic subgroup and $\ker(\varphi)$ is normal in $F/\normclsr{w^n}$, $\comm{\ker(\varphi)}$ is a normal subgroup of $F/\normclsr{w^n}$.
    Thus, let
    \[
    Q'=(F/\normclsr{w^n})/\comm{\ker(\varphi)}
    \]

    Since $\ker(\varphi)$ is a subgroup of $F/\normclsr{w^n}$, $\Ab(\ker(\varphi))$ is a subgroup of $Q'$.
    Thus, there is a natural inclusion map $\Ab(\ker(\varphi))\rightarrow Q'$.

    Now we have the following short exact sequence:
    \[
        1 \rightarrow \Ab(\ker(\varphi))\rightarrow Q' \rightarrow Q \rightarrow 1
    \]
    Since $\Ab(\ker(\varphi))$ is solvable (Abelian) and $Q$ is solvable, $Q'$ is the extension of a solvable group by a solvable group.
    Thus, $Q'$ is solvable, and so $F/\normclsr{w^n}$ is free-by-solvable.
\end{proof}

We will also use Fox derivatives.
The Fox derivative is an algebraic object which captures geometric information about commutator subgroups.

\begin{definition}[\cite{fox_free_1953}]
    Let $F_X$ be a free group.
    The Fox derivative is a function $X \rightarrow (F_X \rightarrow \mathbb Z[F_X])$, written $\fd{w}{x}$ for $w\in F_X$ and $x\in X$.
    It is defined to follow the following laws:
    \begin{align}
        \fd{x}{x} &= 1 & \\
        \fd{y}{x} &= 0 & \text{ where } y\in X\setminus\{x\} \\
        \fd{\varepsilon}{x} &= 0 & \label{eqn:Foxidentity} \\
        \fd{uv}{x} &= \fd ux + u\fd vx & \tag{Product rule} \label{eqn:Foxprod}
    \end{align}
\end{definition}

Most useful to us is a theorem of Fox:

\begin{theorem}[\cite{myasnikov_word_2010}]\label{thm:Fox}
    Let $F_X$ be a free group.
    Let $\varphi : F_X \rightarrow G$ be a group homomorphism, and let $\varphi^\dagger : \mathbb Z[F_X] \rightarrow \mathbb Z[G]$ be its natural lift to the group ring.
    Then for every $w\in F_X$,
    \[
    \left(\forall x\in X : \varphi^\dagger\left(\fd{w}{x}\right) = 0 \right) \iff w\in \comm{\ker(\varphi)}
    \]
\end{theorem}

\section{Results}

\begin{lemma}\label{lem:fox deriv of linton}
    Let $w$ be a word in a free group $F$ and $a$ be a generator of $F$.
    Let $r$ and $k\geq 1$ be such that:
    \begin{equation}\label{eqn:w r relation}
      w\in r^k\left[\normclsr{r},\normclsr{r}\right]
    \end{equation}
    Let $\varphi_r : F \rightarrow F/\normclsr{r}$ be the quotient map.
    Then
    \[
    \varphi_r\left(\fd{w}{a}\right) = k\varphi_r\left(\fd{r}{a}\right)
    \]
\end{lemma}
\begin{proof}
    Let $c\in \comm{\normclsr{r}}$ be the solution to the equation $w=r^kc$ implied by \autoref{eqn:w r relation}.
    We apply the Fox derivative with respect to $a$, and then the quotient map $\varphi_r$ (lifted to the group ring).
    \begin{align*}
    \varphi_r\left(\fd{w}{a}\right)
    &= \varphi_r\left(\fd{r^kc}{a}\right) \\
    &= \varphi_r\left(\fd{r^k}{a}\right)+\varphi_r\left(r^k\right)\varphi_r\left(\fd{c}{a}\right) \tag{\ref{eqn:Foxprod}}\\
    &= \varphi_r\left(\fd{r^k}{a}\right)+1\cdot 0 \tag{\autoref{thm:Fox}}\\
    \end{align*}
    Iterating the \hyperref[eqn:Foxprod]{product rule of Fox derivatives}, we see that
    \begin{align*}
    \varphi_r\left(\fd{r^k}{a}\right)
    &= \varphi_r\left(\sum_{i=0}^{k-1} r^i\fd{r}{a}\right) \\
    &= k\varphi_r\left(\fd{r}{a}\right)
    \end{align*}
\end{proof}

\begin{notation}
    We will use ``quandle notation'' for denoting group conjugation.
    \[
      y\lhd x := yxy^{-1}
    \]
\end{notation}

\begin{lemma}\label{lem:conjugation of Fox derivative}
    Let $F_X$ be a free group and $\varphi : F_X \rightarrow G$ be a homomorphism.
    If $w\in \ker(\varphi)$ then
    \[
    \varphi\left(\fd{q\lhd w}{a}\right) = \varphi\left(q\fd{w}{a}\right)
    \]
    for all $q\in F_X$ and $a\in X$.
\end{lemma}
\begin{proof}
    \begin{align*}
        \varphi\left(\fd{qwq^{-1}}{a}\right)
        &= \varphi\left(\fd{q}{a}+q\fd{w}{a} + qw\fd{q^{-1}}{a}\right) \tag{\ref{eqn:Foxprod}}\\
        &= \varphi\left(\fd{q}{a}+q\fd{w}{a} + q\fd{q^{-1}}{a}\right) \\
        &= \varphi\left(\fd{qq^{-1}}{a}+q\fd{w}{a}\right) \tag{\ref{eqn:Foxprod}}\\
        &= \varphi\left(q\fd{w}{a}\right) \tag{\ref{eqn:Foxidentity}}
    \end{align*}
\end{proof}

\begin{lemma}\label{lem:rispp}
    Let $w$ be a partially positive word in a free group $F$.
    Let $r\in F$ and $k\geq 1$ be such that $w\in r^k\comm{\normclsr{r}}$.
    Then $r$ is conjugate to a word $r'$ such that the positive generator in $w$ appears only positively in $r'$.
\end{lemma}
\begin{proof}
    Again, let $\varphi_r : F \rightarrow F/\normclsr{r}$ be the natural quotient map, and let $a$ be the generator of $F$ which appears only positively in $w$.

    If $r$ is not cyclically reduced, then there is some $q$ such that $q \lhd r$ is cyclically reduced.
    Let $r'=q\lhd r$.
    Thus, $q\lhd w\in r'^k\comm{\normclsr{r'}}$, for partially positive $w$ and cyclically reduced $r'$.
    From here we can derive the following:
    \begin{align*}
        k\varphi_r\left(\fd{r'}{a}\right)
        &= \varphi_r\left(\fd{q\lhd w}{a}\right)\tag{\autoref{lem:fox deriv of linton}}\\
        &= \varphi_r\left(q\fd{w}{a}\right)\tag{\autoref{lem:conjugation of Fox derivative}}
    \end{align*}
    Assume that $r'$ contains $a^{-1}$, and therefore can be written as a product $r_0a^{-1}r_1$, without cancellation, for some $r_0$ and $r_1$.
    \begin{align*}
        \varphi_r\left(\fd{r_0a^{-1}r_1}{a}\right)
        &= \varphi_r\left(\fd{r_0a^{-1}}{a}\right) + \varphi_r\left(r_0a^{-1}\fd{r_1}{a}\right) \\
        &= \varphi_r\left(\fd{r_0}{a}\right) -\varphi_r(r_0a^{-1}) + \varphi_r\left(r_0a^{-1}\fd{r_1}{a}\right)
    \end{align*}
    We now have a negative term, $-\varphi_r\left(r_0a^{-1}\right)$.
    Now since $\varphi_r\left(q\fd{w}{a}\right)$ contains only positive coefficients, and $-\varphi_r(r_0a^{-1})$ is a negative coefficient, there must be some term elsewhere which cancels it.
    In other words, $\varphi_r\left(\fd{r_1}{a}\right)$ must contain a positive constant term.
    Thus $r_1$ must have a prefix of $sa$ (without cancellation) where $\varphi_r(s)=1$.
    However, $s$ must be a proper subword of $r$ and a classic result of Weinbaum, \cite{weinbaum_relators_1972}*{Thm.\ 2} shows that no proper non-empty subword of $r$ can be trivial under $\varphi_r$.
    Thus $r$ cannot contain $a^{-1}$ terms.
\end{proof}
\begin{lemma}\label{lem:w=r}
    Let $F$ be a free group.
    Let $w\in F$ be partially positive and not a proper power.
    If $r$ and $k\geq 0$ are such that
    \[
        w\in r^k\comm{\normclsr{r}}
    \]
    then $r$ is conjugate to $w$.
\end{lemma}
\begin{proof}
    Once again, let $\varphi_r : F \rightarrow F/\normclsr{r}$ be the natural quotient map, and let $a$ be the generator of $F$ which appears only positively in $w$.
    Let $r=q\lhd r'$ where $r'$ is partially positive by \autoref{lem:rispp} and the exponent sum of $a$ in $q$, which we denote $|q|_a$, is $0$ modulo $|r'|_a$.
    Such a $q$ must exist for every modular residue by the following argument:
    Suppose we have some $q$ with exponent sum $t$.
    Let $r_0a$ be the shortest prefix of $r'$ containing an $a$.
    Then
    \begin{align*}
        r & =qr'q^{-1} \\
          &= qr_0aa^{-1}r_0^{-1}r'r_0aa^{-1}r_0^{-1}q^{-1} \\
          &= qr_0a \lhd (a^{-1}r_0^{-1} \lhd r')
    \end{align*}
    $a^{-1}r_0^{-1} \lhd r'$ is partially positive, and $|qr_0a|_a=t+1$.
    Then by induction, every modular residue is possible.
    In the same way we can stipulate that $r'$ ends with $a$.
    If it does not, simply cycle it until it does.

    Let us assume without loss of generality that $w$ also ends in $a$.
    We will express $w$ and $r'$ as products:
    \begin{align*}
    w &= \prod_{i=0}^{nk-1} w_ia \\
    r' &= \prod_{i=0}^{n-1}r_ia
    \end{align*}
    where $n=|r|_a$ and each $w_i$ and $r_i$ does not contain $a$ (or $a^{-1}$).
    $|w|_a=nk$, because $w=r^kc$ for some $c\in \comm{\normclsr{r}}\subseteq \comm{F}$.
    Their Fox derivatives follow:
    \begin{align*}
        \fd{w}{a} &= \sum_{i=0}^{nk-1}w_0\prod_{j=1}^i aw_j 
        \\
        \fd{r'}{a} &= \sum_{i=0}^{n-1}r_0\prod_{j=1}^i ar_j 
    \end{align*}
    For brevity, we will let the $i$th term of the first sum be $\hat w_i$, and likewise $\hat r_i$ is the $i$th term of the second sum.

    We can derive the following relationship between the Fox derivatives of $w$ and $r'$:
    \begin{align*}
      \varphi_r\left(\fd{w}{a}\right)
      & = k\varphi_r\left(\fd{r}{a}\right) \tag{\autoref{lem:fox deriv of linton}} \\
      & = k\varphi_r\left(\fd{q\lhd r'}{a}\right) \\
      & = k\varphi_r\left(q\fd{r'}{a}\right) \tag{\autoref{lem:conjugation of Fox derivative}}
    \end{align*}
    Each application of the defining relation $r=1$ on some word changes its exponent sum on $a$ by $n$.
    Since $|\hat w_i|_a=|\hat r_i|_a = i$ and $|q|_a \equiv 0 \mod n$,
    $\varphi_r(\hat w_i)=\varphi_r(q\hat r_j)$ only if $i\equiv j \mod n$.
    Since there are only $n$ terms in $\fd{r'}a$, each has a unique exponent sum of $a$ modulo $n$.
    Thus $\varphi_r$ of each term in $\fd{r'}a$ must be unique.
    For each $i < n$, there must be exactly $k$ terms in $\varphi_r\left(\fd wa\right)$ equal to $\varphi_r\left(q\hat r_i\right)$.
    There are exactly $k$ terms $\hat w_j$ where $j$ has any particular modular residue, thus 
    \begin{equation*}\label{eqn:wjqri}
    \varphi_r(\hat w_j)=\varphi_r(q\hat r_i) \iff i \equiv j \mod n
    \end{equation*}

    Now consider $w_j$ and $r_i$ for some $1 \leq j \leq nk-1$ and $i\equiv j \mod n$.
    By definition, $\hat w_j = \hat w_{j-1}aw_j$; thus $w_j=a^{-1}\hat w_{j-1}^{-1}\hat w_j$.
    Similarly, $r_i=a^{-1}\hat r_{i-1}^{-1}\hat r_i$.
    Thus,
    \begin{align*}
    \varphi_r(w_j)
    &= \varphi_r(a^{-1}\hat w_{j-1}^{-1}\hat w_j) \\
    &= \varphi_r(a^{-1}(q\hat r_{i-1})^{-1}(q\hat r_i)) \\
    &= \varphi_r(a^{-1}\hat r_{i-1}^{-1}\hat r_i) \\
    &= \varphi_r(r_i)
    \end{align*}
    Since neither $w_j$ nor $r_i$ contain $a$ by definition, Freiheitssatz implies that $w_j=r_i$ when $0\neq j \equiv i \mod n$.
    It still remains to show that this is the case when $j=i=0$.
    To do this we will rotate $w$ by conjugating by $w_{nk-1}a=r_{n-1}a$.
    This will move $w_0$ into the second position, and then we will be able to use the above technique to show the desired claim.
    Naturally, since $w=r^kc$ for some $c\in \comm{\normclsr{r}}$, after conjugation we have $w_{nk-1}a\lhd w= (r_{n-1}a\lhd r)^k\tilde c$ for some $\tilde c\in \comm{\normclsr{r}}$.
    Recall $r=q\lhd r'$ where $r'$ ends with $a$ and $|q|_a\equiv 0\mod{|r|_a}$.
    Thus, letting $\tilde w=w_{nk-1}a\lhd w$, $\tilde r' = r_{n-1}a\lhd r'$, and $\tilde q = r_{n-1}a\lhd q$,
    \[
    \tilde w = (\tilde q\lhd \tilde r')^k\tilde c
    \]
    This satisfies all the necessary conditions.
    Specifically:
    \begin{itemize}
        \item $\tilde w \in (\tilde q\lhd \tilde r')^k\comm{\normclsr{\tilde r'}}$.
        \item $\tilde w$ is partially positive, since it is a rotation of $w$.
        \item $\tilde w$ ends with $a$.
        \item $\tilde r'$ is partially positive, since it is a rotation of $r'$.
        \item $\tilde r'$ ends with $a$.
        \item $|\tilde q|_a\equiv 0 \mod{|\tilde r|_a}$ since $\tilde q$ is conjugate to $q$ and $\tilde r$ is conjugate to $r$ and $|q|_a\equiv 0 \mod |r|_a$.
    \end{itemize}
    With these assumptions, the proof above shows that $w_0=r_0$.

    Since $w$ is not a proper power, $k=1$ and $w=r'$.
    $r'$ is conjugate to $r$ so $w$ is conjugate to $r$.
\end{proof}

Now we can prove our main result.

\begin{proof}[Proof of \autoref{thm:real main}]
    This follows from \autoref{lem:w=r}, \autoref{thm:linton main}, and \autoref{thm:linton torsion}
\end{proof}

\section{Shpilrain's question}
We have shown that it is not possible to use \autoref{thm:Baumslag} to provide a counter-example to \autoref{q:Shpilrain's question}.
We would like to provide an answer to this question using a different route.
The following argument is due to Carl-Fredrik Nyberg-Brodda and presented by the authors here for completeness:

\begin{definition}
   Consider a word $w\in F_r$.
   The {\it Fox ideal} of $w$, $\mathfrak f_w\subseteq \mathbb Z[F_r]$ is the ideal generated by the Fox derivatives of $w$.
   \[
   \mathfrak f_w =
   \left(\fd{w}{x_1},\fd w{x_2},\dots,\fd w{x_r}\right)
   \]
\end{definition}

\begin{lemma}\label{lem:fox ideal auto}
   The Fox ideal of $w\in F_r$ is equivariant under $\Aut(F_r)$.
   That is, for all $\varphi \in \Aut(F_r)$
   \[
   \mathfrak f_{\varphi(w)} = \varphi(\mathfrak f_w)
   \]
\end{lemma}
This follows directly from \cite{fox_free_1953}*{2.6}.

Now we will give a useful interpretation of the Fox derivative.
Consider the Fox derivative $\fd w x$, and some quotient map $\varphi(F_r)\to G$.
Then $\varphi\left(\fd w x\right)$, being an element of the group ring $\mathbb Z[G]$, assigns an integer co\"efficient to every element of $G$.
There is a correspondence between elements of $G$ and $x$ edges in the Cayley graph of $G$;
each edge simply corresponds to its initial vertex.
Now we notice that the number that $\varphi\left(\fd w x\right)$ assigns to an edge is the number of times the path traced by $w$ passes over this edge in the positive direction, minus the number of times it passes in the negative direction.
For a positive word, this is particularly simple: it is just the number of times $w$ crosses a particular edge.

Now we consider what happens when $w$ is a positive word and $\varphi=\Ab$.
\begin{observation}\label{obs:positivepath}
  Let $w\in F_r$ be a positive word.
  A positive path in $\mathbb Z^r$ visits any vertex at most once, and only crosses edges in the positive direction.
  Thus the sum of the Fox derivatives $\Ab\left(\sum_{x\in X}\fd w x\right)$ is exactly the vertices visited by the path of $w$ excluding the final vertex at $\Ab(w)$.
\end{observation}

\begin{proposition}[Carl-Fredrik Nyberg-Brodda, personal communication]\label{CFNB example}
    The word $w=a^n[b,c]^m\in F(a,b,c)$ for $n,m\geq 2$ is partially positive, but not potentially positive.
\end{proposition}
\begin{proof}
  The first part is trivial.
  $a$ appears only positively in $w$ thus $w$ is partially positive.

  Now we compute the Fox derivatives of $w$:
  \begin{align*}
      \fd w a&=\sum_{i=0}^{n-1} a^i\\
      \fd w b&=a^n(1-bcb^{-1})\left(\sum_{i=0}^{m-1}[b,c]^i\right)\\
      \fd w c&=a^nb(1-cb^{-1}c^{-1})\left(\sum_{i=0}^{m-1}\left[c,b^{-1}\right]^i\right)
  \end{align*}
  We notice that if we apply the Abelianization map to $\fd wb$ and $\fd wc$ we get a common factor of $m$ in our co\"efficients.
  \begin{align*}
      \Ab\left(\fd w b\right)&=m\cdot a^n(1-c)\\
      \Ab\left(\fd w c\right)&=m\cdot a^nb(1-b^{-1})
  \end{align*}
  If we take our co\"ordinates modulo $m$, this gives a factor of 0 and trivializes these derivatives.
  The co\"efficients in $\fd wa$ are all 1.
  So, since $m\geq 2$, this derivative remains non-trivial.
  This makes our Fox ideal in $\mathbb Z/m\mathbb Z\left[\mathbb Z^3\right]$ rather simple:
  \begin{align*}
      [\mathfrak f_w]&=\left(\sum_{i=0}^{n-1} \Ab(a)^i\right)
  \end{align*}

  Now we consider the Fox ideal $\mathfrak f_{\varphi(w)}$ for some $\varphi\in\Aut(F_3)$.
  By \autoref{lem:fox ideal auto}, $\mathfrak f_{\varphi(w)}=\varphi\left(\mathfrak f_w\right)$.
  Since the commutator subgroup is a characteristic subgroup of $F_3$, $\varphi$ sends elements of the commutator subgroup to elements of the commutator subgroup.
  Thus if we Abelianize $\mathfrak f_{\varphi(w)}$, we get the same factors of $m$ in the $b$ and $c$ derivatives.
  Thus the Fox ideal in $\mathbb Z/m\mathbb Z[\mathbb Z^3]$ is
  \begin{align*}
      [\mathfrak f_{\varphi(w)}]&=\left(\sum_{i=0}^{n-1} \Ab(\varphi(a))^i\right)
  \end{align*}

  We can suppose, by way of contradiction, that $\varphi(w)$ is positive.
  The astute reader may already begin to see an issue with this.
  To draw out the issue, we will apply another quotient map to our ideal.
  This time we quotient $\mathbb Z^3$ by $\Ab(\varphi(w))$, which is equal to $\Ab(\varphi(a))^n$.
  Call this new group $G^{\Ab}_{\varphi(w)}$, and its group ring with $\mathbb Z/m\mathbb Z$ co\"efficients $R$.

  Elements of our ideal $[\mathfrak f_{\varphi(w)}]\subseteq R$ are invariant under multiplication by $\Ab(\varphi(a))$, since the ideal is generated by $\sum_{i=0}^{n-1} \Ab(\varphi(a))^i$, which is invariant under this translation.

  Now consider the sum of the Fox derivatives of $\varphi(w)$ taken in $R$.
  Before our quotient by $\Ab(\varphi(w))$, the sum of the Fox derivatives is exactly the vertices (excluding the final vertex) of the path of $\varphi(w)$ in the Cayley graph of $\mathbb Z^3$.
  Since $\varphi(w)$ is positive, it is clear that no two vertices visited can differ by $\Ab(\varphi(w))$ (or a multiple thereof).\footnote{This also follows from \cite{weinbaum_relators_1972}*{Thm.\ 2} via a somewhat simple argument.}
  Thus, the sum of Fox derivatives of $\varphi(w)$ taken in $R$ is exactly the vertices in the Cayley graph of $G_{\varphi(w)}^{\Ab}$ traversed by $\varphi(w)$ (counting the initial and final vertex only once).
  However, the sum of the Fox derivatives is a member of the ideal.
  Thus the path of $\varphi(w)$ must be translation invariant under $\Ab(\varphi(a))$.

  Thus $\varphi(w)$ is a perfect $n$th power.
  However, $w$ is not a perfect $n$th power.
  Thus $\varphi(w)$ cannot be in the automorphic image of $w$.
  This is a contradiction.
  Thus, our assumption is false, and $w$ is not potentially positive.
\end{proof}

\section{Further directions}

We have shown that for every partially positive one-relator group $G$ there is some $n$ such that $G^{(n)}$ is free.
We call $G$ ``free-by-$n$-solvable''.
This is an isomorphism invariant, and it has the potential to distinguish non-isomorphic partially positive one-relator groups.

\begin{question}
    When is a partially positive one-relator group free-by-$n$-solvable?
\end{question}

This question would be particularly compelling if partially positive words were very common.
While partially positive words greatly extend \hyperref[thm:Baumslag]{Baumslag's result on positive words}, we will see shortly that the asymptotic probability that a word is partially positive is 0.
That is in some sense, long partially positive words are very rare.

To show this we will introduce a concept.
\begin{definition}
    The {\it exponential growth rate} of a set of reduced words $S\subseteq F_r$, is
    \[
    \inf\{b\in \mathbb R : g(S)\in O(b^n)\}
    \]
    where $g(S)$ is the counting function of $S$:
    \[
        g(S) = n \mapsto |\{s\in S: |s| \leq n\}|
    \]
\end{definition}
The growth rate is closely related to the notion of topological entropy in symbolic dynamics.
Clearly if the growth rate of $S$ is less than $2r-1$ (the growth rate of $F_r$), then the probability of being in $S$ converges to 0 exponentially.

\begin{theorem}
    The set of partially positive words of length $n$ in the free group of rank $r$ has an exponential growth rate of $\sqrt{r^2-2r+2}+r-1 < 2r-1$.
\end{theorem}
\begin{proof}
    The exponential growth rate can be obtained by solving a recurrence relation.
    Let $a\in X_r$ be the positive generator.
    Since $a$ is the only letter with a special rule all other letters behave identically.
    An $a$ can follow any letter, $2r-2$ non-$a$ letters can follow an $a$, and $2r-3$ non-$a$ letters can follow a non-$a$ letter (every non-$a$ letter but its inverse).
    We can encode this as a matrix.
    \[
        \begin{bmatrix}
            2r-3 & 1 \\
            2r-2 & 1
        \end{bmatrix}
    \]
    The largest positive eigenvalue for this matrix is $\sqrt{r^2-2r+2}+r-1$ which gives us the exponential growth rate.
\end{proof}

While partially positive words have asymptotic probability 0, $\sqrt{r^2-2r+2}+r-1$ exceeds previous lower bounds for the growth rate of potentially positive words in $F_r$ when $r\geq 3$.
To our knowledge, this constitutes an improvement on the lower bound on the exponential growth rate of words relating a residually solvable group.

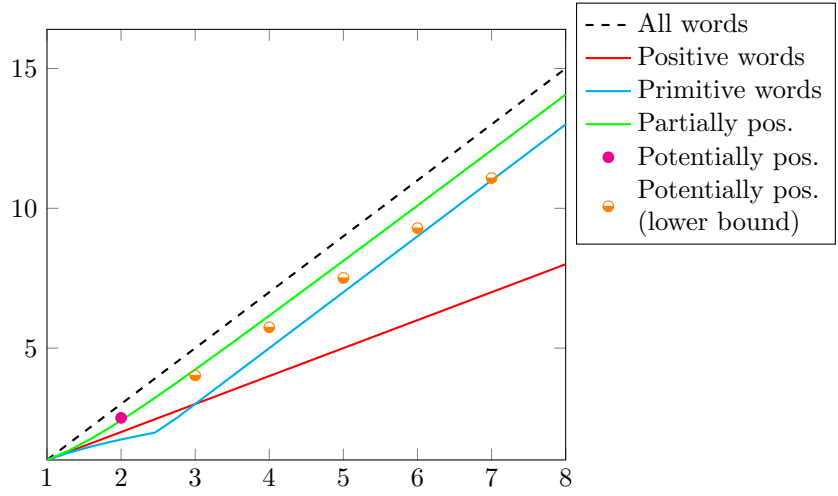
\begin{figure}
    \centering
    \def\size{8}
    \begin{tikzpicture}
        [ domain=1:\size
        ]
%

        \begin{axis}
            [ xmin=1
            , xmax=\size
            , ymin=1
            , legend style=
                { at={(1.02, 0.5)}
                , anchor=south west
                , align=left
                }
            , legend cell align={left}
            ]

        \addplot
            [ dashed
            , thick
            ]
            {2*\x-1};
        \addlegendentry[align=left]{All words}
        \addplot
            [ domain=1:\size
            , color=red
            , thick
            ]
            {\x};
        \addlegendentry{Positive words}
        \addplot
            [ domain=1:\size
            , color=cyan
            , thick
            ]
            {max(sqrt(2*\x-1),2*\x-3)};
        \addlegendentry{Primitive words}
        \addplot
            [ green
            , thick
            ]
            {sqrt(\x^2-2*\x+2)+\x-1};
        \addlegendentry{Partially pos.}
        \addplot
            [ mark=*
            , color=magenta
            , only marks
            ]
            coordinates
                { (2,2.505)
                };
        \addlegendentry{Potentially pos.}
        \addplot
            [ mark=halfcircle
            , color=orange
            , mark color=orange
            , only marks
            ]
            coordinates
                { (3,4.024)
                  (4,5.746)
                  (5,7.509)
                  (6,9.290)
                  (7,11.083)
                };
        \addlegendentry{Potentially pos.\\ (lower bound)}
        \end{axis}
    \end{tikzpicture}
    \caption{
    The exponential growth rates of various families of words which give residually solvable one-relator groups.
    }
    \label{fig:placeholder}
\end{figure}

\begin{definition}
    We call a word in $F_X$ potentially partially positive, or 3P, if it is automorphic to a partially positive word.
\end{definition}

\begin{example}
    The word $w=a^{-1}bac^{-1}a^2b^{-1}a^{-2}c^4\in F(a,b,c)$ is 3P but not partially positive.
\end{example}
\begin{proof}
    Consider the automorphism $\varphi = a\mapsto a;b\mapsto b;c\mapsto a^2ca$.
    $\varphi(w) = [b,c^{-1}](a^3c)^4$ which has only positive exponents on $a$.
\end{proof}

Naturally, we ask the same questions of 3P as have already been asked for potentially positive words:

\begin{question}
    Is there an algorithm to decide whether an input word is 3P?
\end{question}

\begin{question}
    What is the asymptotic probability that a word in $F_r$ is 3P?
    What is the exponential growth rate of a 3P word in $F_r$?
\end{question}

\begin{question}
    Is there a ``3P-blocking word'', that is a word which is not a subword of any cyclically reduced 3P word in $F_r$?
\end{question}

If the answer to this question is ``yes'', then the asymptotic probability that a word is 3P is $0$.

All of these questions have been answered for rank 2 potentially positive words, and since 3P is potential positivity in $F_2$, they are settled for 3P.

Both deciding 3P and potential positivity are special cases of a more general problem.
Positive words form a subsemigroup of the free group, and while all partially positive words do not form a subsemigroup, if we fix a generator $a$ and take the subset of partially positive words for which $a$ is the positive generator we get a subsemigroup whose automorphic closure is 3P.

\begin{problem}
    Given a subsemigroup $M\subseteq F_r$ and a word $w\in F_r$ determine whether $w$ is in the automorphic closure of $M$, $\Aut(F_r)M$.
\end{problem}

Of course, in general this problem is hopeless.
$M$ need not be finitely generated, and indeed the subsemigroup whose automorphic closure is 3P is not finitely generated for $r\geq 2$.
However, if we were able to decide the automorphic closure of the monoid $K$ generated by $X^\pm/\left\{a^{-1}\right\}$ in $F_X$, it would go a long way to deciding 3P.
$K$ contains the subsemigroup whose automorphic closure is 3P, and their difference is just words which do not contain $a$ or $a^{-1}$.
There are efficient algorithms for deciding whether a word is automorphic to a word not containing all the generators, and such words are asymptotically negligible.

\begin{question}
    Is there a finitely generated subsemigroup $M\subseteq F_r$ such that its automorphic closure in $F_r$ is not a decidable set?
\end{question}

We conjecture that the answer is ``yes''.

\begin{question}
    Is there a finitely generated subsemigroup $M\subseteq F_r$ such that $\mathbb P^\infty_r(M) = 0$ but $\mathbb P^\infty_r(\Aut(F_r)M) > 0$?
\end{question}

We conjecture that the answer is ``no''.

\section{Acknowledgment}
The authors would like to thank Carl-Fredrik Nyberg Brodda for pointing out errors in an early version of this manuscript, and for providing a solution to Shpilrain's question.

\bibliography{refs}

\end{document}